\documentclass[a4paper, 12pt]{amsart}

\usepackage[a4paper]{geometry}
\usepackage{amssymb}
\usepackage{amsthm}
\usepackage{amsmath}
\usepackage{graphicx}
\usepackage{hyperref}
\usepackage{multirow}
\usepackage{array}
\usepackage[colorinlistoftodos]{todonotes}
\usepackage[english]{babel}
\usepackage{lmodern}
\usepackage{comment}
\usepackage{bbold}
\usepackage{mathtools}
\usepackage{enumerate}
\usepackage{cite}
\usepackage{extpfeil}
\usepackage[OT2,T1]{fontenc}

\newcommand{\genlegendre}[4]{%
	\genfrac{(}{)}{}{#1}{#3}{#4}%
	\if\relax\detokenize{#2}\relax\else_{\!#2}\fi
}
\newcommand{\legendre}[3][]{\genlegendre{}{#1}{#2}{#3}}
\newcommand{\CT}[2]{\langle #1,#2\rangle_{CT}}

\allowdisplaybreaks

\usepackage[utf8]{inputenc}
\usepackage[T1]{fontenc}

\newtheorem{theorem}{Theorem}[section]
\newtheorem{lemma}[theorem]{Lemma}
\newtheorem{corollary}[theorem]{Corollary}
\newtheorem{conjecture}{Conjecture}

\theoremstyle{definition}
\newtheorem{definition}[theorem]{Definition}
\newtheorem{proposition}[theorem]{Proposition}

\newtheorem*{question}{Question}

\newtheorem{example}[theorem]{Example}

\theoremstyle{remark}
\newtheorem{remark}[theorem]{Remark}
\newcommand{\longcomment}[1]{}

\DeclareMathOperator{\ord}{ord}

\DeclareMathOperator{\Sel}{Sel^{(2)}}

\DeclareMathOperator{\rank}{rank}
\DeclareMathOperator{\loc}{\textrm{loc}}

\DeclareSymbolFont{cyrletters}{OT2}{wncyr}{m}{n}
\DeclareMathSymbol{\Sha}{\mathalpha}{cyrletters}{"58}

\DeclareMathOperator{\Hom}{Hom}

\DeclareMathOperator{\Gal}{Gal}
\DeclareMathOperator{\GalQ}{Gal(\overline{\mathbb{Q}}/\mathbb{Q})}

\newcommand{\T}{\tilde{T}}

\newcommand{\whichbold}[1]{\mathbb{#1}} 

\newcommand{\ZZ}{\whichbold{Z}}
\newcommand{\FF}{\whichbold{F}}
\newcommand{\RR}{\whichbold{R}}

\newcommand{\QQ}{\mathbb{Q}}

\newcommand{\Fp}{\whichbold{F}_{p}}

\subjclass[2010]{14H52,   11G05}

\keywords{elliptic curves, quadratic twists, Selmer groups}

\numberwithin{equation}{section}

\author{Matija Kazalicki}
\address{Department of Mathematics\\ 
	University of Zagreb\\
	 Bijeni\v{c}ka cesta 30\\
	  10000 Zagreb\\
	  Croatia}
\email{matija.kazalicki@math.hr}

\title[Quadratic twists of genus one curves and Diophantine quintuples]{Quadratic twists of genus one curves and Diophantine quintuples}

\makeatletter
\renewcommand{\tocsection}[3]{%
	\indentlabel{\@ifnotempty{#2}{\bfseries\ignorespaces#1 #2\quad}}\bfseries#3}
\renewcommand{\tocsubsection}[3]{%
	\indentlabel{\@ifnotempty{#2}{\ignorespaces#1 #2\quad}}#3}

\newcommand\@dotsep{4.5}
\def\@tocline#1#2#3#4#5#6#7{\relax
	\ifnum #1>\c@tocdepth 
	\else
	\par \addpenalty\@secpenalty\addvspace{#2}%
	\begingroup \hyphenpenalty\@M
	\@ifempty{#4}{%
		\@tempdima\csname r@tocindent\number#1\endcsname\relax
	}{%
		\@tempdima#4\relax
	}%
	\parindent\z@ \leftskip#3\relax \advance\leftskip\@tempdima\relax
	\rightskip\@pnumwidth plus1em \parfillskip-\@pnumwidth
	#5\leavevmode\hskip-\@tempdima{#6}\nobreak
	\leaders\hbox{$\m@th\mkern \@dotsep mu\hbox{.}\mkern \@dotsep mu$}\hfill
	\nobreak
	\hbox to\@pnumwidth{\@tocpagenum{\ifnum#1=1\bfseries\fi#7}}\par
	\nobreak
	\endgroup
	\fi}
\AtBeginDocument{%
	\expandafter\renewcommand\csname r@tocindent0\endcsname{0pt}
}
\def\l@subsection{\@tocline{2}{0pt}{2.5pc}{5pc}{}}
\makeatother

\begin{document}
	\maketitle
	
\begin{abstract}
Motivated by the theory of Diophantine $m$-tuples, we study rational points on quadratic twists $H^d:d y^2=(x^2+6x-18)(-x^2+2x+2)$, where $|d|$ is a prime.
If we denote by $S(X)=\{ d \in \ZZ:  H^d(\QQ)\ne \emptyset, |d| \textrm{ is a prime}\textrm{ and } |d| < X\},$ then, by assuming some standard conjectures about the ranks of elliptic curves in the family of quadratic twists, we prove that as $X \rightarrow \infty$ 
$$\frac{43}{256}+o(1)\le \frac{\#S(X)}{2\pi(X)}\le \frac{46}{256}+o(1).$$

\end{abstract}	

\section{Introduction}
For an integer $d$, a set of $m$ distinct nonzero rational numbers
with the property that the product of any two of its distinct
elements plus $d$ is a square is called a rational Diophantine $m$-tuple with the property $D(d)$ or $D(d)$-$m$-tuple. The $D(1)$-$m$-tuples (with rational elements) are called simply rational Diophantine $m$-tuples and have been studied since ancient times, starting with Diophantus, Fermat, and Euler.

It is not known how large can a rational Diophantine tuple be.
Dujella, Kazalicki, Miki\'c, and Szikszai \cite{Dujella_Kazalicki_Mikic_Szikszai} proved that there are infinitely many rational Diophantine sextuples, while no example of a rational Diophantine septuple is known. Also, no example of rational $D(d)$-sextuple is known if $d$ is not a perfect square. For more information on Diophantine $m$-tuples see the survey article \cite{Dujella_Notices}.

We are interested in the following question.

\begin{question}
	Does there exist a rational $D(d)$-quintuple for every $d\in \ZZ$?
\end{question}

Dujella and Fuchs \cite{Dujella_Fuchs} proved that there are infinitely many squarefree integers $d$'s for which there are infinitely many rational $D(d)$-quintuples, and Dra\v zi\'c \cite{Drazic} (improving the similar result from \cite{Dujella_Fuchs}) proved, assuming the Parity conjecture for the quadratic twists of several explicitly given elliptic curves, that for at least $99.5\%$ of squarefree integers $d$ there are infinitely many rational $D(d)$-quintuples.

Following an idea from \cite{Dujella_Fuchs}, we start with a $D(\frac{16}{9}x^2(x^2-x-3)(x^2+2x-12))$-quintuple in $\ZZ[x]$ 
	\footnotesize\begin{equation*}
	{  \left\{\frac{1}{3}(x^2+6x-18)(-x^2+2x+2),\frac{1}{3} x^2(x+5)(-x+3),(x-2)(5x+6),\frac{1}{3}(x^2+4x-6)(-x^2+4x+6),4x^2\right\}}
\end{equation*}\normalsize
found by Dujella \cite{Dujella_Fib} (and used to prove that there are infinitely many $D(-1)$-quintuples in \cite{Dujella_Fib_2}).
Note that for rational $u\ne 0$, if $\{a,b,c,d,e\}$ is $D(qu^2)$-quintuple, then $\{\frac{a}{u},\frac{b}{u},\frac{c}{u},\frac{d}{u},\frac{e}{u}\}$ is $D(q)$-quintuple.
In particular, for squarefree integer $d$, if $$d y^2=(x^2-x-3)(x^2+2x-12)$$ for some $x,y\in \QQ$ then by dividing the elements of quintuple above with $\frac{4}{3}x y$ we obtain $D(d)$-quintuple. Thus, if the equation above has infinitely many solution, we may conclude that there are infinitely many $D(d)$-quintuples.

Consider the genus one quartic 
$$H:\quad y^2=(x^2-x-3)(x^2+2x-12).$$
For a squarefree integer $d$, we denote by $H^d: d y^2=f(x)$ the quadratic twist of $H$ with respect to $\QQ(\sqrt{d})$. Quartic $H$, as a (singular) genus one curve with a rational point at infinity, is birationally equivalent to the elliptic curve $E/\QQ$
 $$E: y^2=(x-9)(x-8)(x+18).$$
 Likewise, we denote by $E^d$ the quadratic twist of $E$ by $\QQ(\sqrt{d})$. Thus $H^d(\QQ) \ne \emptyset$ implies that $H^d$ is birationally equivalent to $E^d$. Since, by Proposition \ref{prop:infinity}, $H^d(\QQ)\ne \emptyset$ implies that $H^d(\QQ)$ is infinite and consequently that there are infinitely many $D(d)$-quintuples, we are led to the study of squarefree integers $d$ for which $H^d(\QQ)\ne \emptyset$.

 In this paper we will focus on twists by $\QQ(\sqrt{d})$ where $|d|$ is prime. Let
$$S=\{ d \in \ZZ:  H^d(\QQ)\ne \emptyset \textrm{ and } |d| \textrm{ is a prime}\}.$$ 
\begin{question}
	 What is asymptotically the size of set $S(X)=\{d \in S: |d|< X\}$ as $X \rightarrow \infty$?
\end{question}

Surprisingly, and in contrast with the analogous problem for the quadratic twists of elliptic curves, not much is known about this question.

\c{C}iperiani and Ozman gave a criterion for the set of rational points of the quadratic twist of quartic to be non-empty in terms of the image of the global trace map $tr_{\QQ(\sqrt{d})/\QQ}$ on an elliptic curve (see Section 2 of \cite{Ciperiani_Ozman}), but in general, no estimates for the size of set $S(X)$ are known.

For a squarefree $d$, the quartic $H^d$, as a $2$-covering of $E^d$, represents an element of $\Sel(E^d)$, the $2$-Selmer group of $E^d$, provided that $H^d$ is everywhere locally solvable (i.e. $H^d(\QQ_v)\ne \emptyset$ for all places $v$ -- we write ELS for short). For the interpretation of Selmer group elements as $2$-covers of $E$ see Section 1.2 of \cite{Stoll_descent}. 

If $|d|=p$ is a prime, then Proposition \ref{prop:ELS} implies that $H^d$ is ELS if and only if $\legendre{p}{13}=1$ or $p=13$. Thus, for such $d$, $H^d(\QQ) = \emptyset$ if and only if $H^d$ represents a nontrivial element in $\Sha(E^d)[2]$ (where $\Sha(E^d)$ denotes the Tate-Shafarevich group of $E^d$), or more precisely,  if and only if the image of $H^d$ under the map $\iota:\Sel(E^d) \rightarrow \Sha(E^d)[2]$ from the exact sequence
\begin{equation}\label{eq:exact}
0\longrightarrow E^d(\QQ)/2E^d(\QQ) \longrightarrow \Sel(E^d) \longrightarrow \Sha(E^d)[2]\longrightarrow 0
\end{equation}
is nonzero. In this case, we say that $H^d$ represents the element of order two in $\Sha(E^d)$.

Our main tool for studying the image of $H^d$ in $ \Sha(E^d)[2]$  is the Cassels-Tate pairing on $\Sha(E^d)$ with values in $\QQ / \ZZ$, or more precisely, its extension to a pairing on  $2$-Selmer group by \eqref{eq:exact}  $$\langle \cdot, \cdot \rangle_{CT}:  \Sel(E^d) \times  \Sel(E^d) \rightarrow \ZZ/2\ZZ=\{0,1\}.$$

This pairing is bilinear, alternating, and non-degenerate on $\Sha(E^d)[2]/2\Sha(E^d)[4]$, or equivalently, on $\Sel(E^d)/2\,\textrm{Sel}^{(4)}(E^d)$ (see Section \ref{sec:governing}). In particular, $\dim_{\FF_2}\Sha(E^d)[2]/2\Sha(E^d)[4]$ is even, thus equal to $0$ or $2$ if $|d|$ is a prime (see Proposition \ref{prop:sha_trivial}). Thus, if we find a class $L\in \Sel(E^d)$ such that $\CT{H^d}{L}=1$, we can conclude that $\iota(H^d)\ne 0$, and, hence, that $H^d$ represents the element of order two in $\Sha(E^d)$. 
If $\Sha(E^d)[2]$ is nontrivial and $\Sha(E^d)[2] = 2\Sha(E^d)[4]$ (see Proposition \ref{prop:sha24}), then we {\bf can not} obtain any information about $H^d$ using this method.

For estimating the asymptotic behaviour of $\#S(X)$ as $X \rightarrow \infty$ we will assume the following ``standard'' conjectures.
\begin{conjecture}\label{conj:1}
	$100\%$ of quadratic twists $E^d$ where $|d|$ is a prime have rank $0$ or $1$.
\end{conjecture}

Note that this conjecture is now a theorem under the BSD conjecture if we let $d$ range over all squarefree integers (see Smith \cite{Smith_22_dist, Smith_22_fixed}).

\begin{conjecture}[The parity conjecture]\label{conj:2}
	For all $d\in \ZZ$ where $|d|$ is prime, $$(-1)^{\rank(E^d)}= w(E^d),$$ where $w(E^d)$ is the root number of the elliptic curve $E^d$.
\end{conjecture}

It follows from Proposition \ref{prop:infinity} that the contribution of $d$'s ($|d|$ is a prime) for which the root number $w(E^d)$ is equal to $1$ to the $\#S(X)$ is negligible since by Conjecture \ref{conj:1} $100\%$ of the curves $E^d$ will have rank $0$ or $1$ and by Conjecture \ref{conj:2} that rank is even, hence zero.

On the other hand, in the case $w(E^d)=-1$, if $\dim_{\FF_2}\Sel(E^d)=3$ (see Proposition \ref{prop:infinity} for the description of $\Sel(E^d)$) then by Conjecture \ref{conj:2} $\rank(E^d)=1$ so $\Sha(E^d)[2]$ is trivial (note that $E^d$ has full rational two torsion, hence $\dim_{\FF_2}\Sha(E^d)[2] =\dim_{\FF_2} \Sel(E^d)-\rank(E^d) - 2 = 0$). 

Hence the only interesting case (in which we expect $\Sha(E^d)[2]$ generically to be nontrivial) is when $\dim_{\FF_2}\Sel(E^d)=5$ or equivalently (see Proposition \ref{prop:sha_trivial}) when $d\in T=T^+ \cup T^-$ where

\begin{align*}
	T^+ &= \{d >0: |d| \textrm{ is prime}, \legendre{d}{13} = 1, \legendre{d}{3}=1, d\equiv 1 \pmod{8}  \},\\
	T^- &= \{d<0: |d| \textrm{ is prime}, \legendre{d}{13}=1, \legendre{d}{2}\cdot\legendre{d}{3}=-1, d\equiv 5,7 \pmod{8}\}.
\end{align*}

Define 
\begin{align}\label{eq:quartics}
	\begin{split}
		H_1&: y^2=4 x^4-56 x^2+169 \in \Sel(E),\\
		H_2&: y^2=18 x^4-24 x^3-32 x^2+40x + 34 \in \Sel(E),\\	
		F_1&: y^2=11 x^4+12 x^3+56x^2+24x+68 \in \Sel(E^{-1}),\\
		F_2&: y^2=x^4+56x^2+676 \in \Sel(E^{-1}).\\
	\end{split}
\end{align}

We show in Proposition \ref{prop:sha_trivial} that if $d\in T$, $\Sel(E^d)$ is generated by  the image of the two torsion $E^d[2]$ under the Kummer map, $H^d$, and by the quadratic twists of those classes in \eqref{eq:quartics} which land in $\Sel(E^d)$. Hence for such $d$'s $\dim_{\FF_2}\Sel(E^d) = 5$. Proposition \ref{prop:ELS_quartics} describes when these twists of quartics in \eqref{eq:quartics} are ELS. Note that this simple explicit description of $\Sel(E^d)$ (see Proposition \ref{prop:sha_trivial}) is the main reason why we considered only quadratic twists by $d$ where $|d|$ is prime. In general, for squarefree $d$, $\dim_{\FF_2}\Sel(E^d)$ is unbounded.

Assuming the parity conjecture for $E^d$, where $d\in T$, we can deduce that $\dim_{\FF_2} \Sha(E^d)[2] = 0$ or $2$. Assume further that $\Sha(E^d)[2] \ne 2\Sha(E^d)[4]$. The non-degeneracy of the Cassels-Tate pairing implies that for $d$ such that $\iota(H^d)\ne 0$ there exists class $L\in \Sel(E^d)$ (also with $\iota(L)\ne 0$) for which  $\CT{H^d}{L}=1$. The following theorem then follows easily from Section \ref{sec:governing}, Proposition \ref{prop:sha_trivial} and the previous discussion.

\begin{theorem}\label{thm:main}
	Let $d\in T$ such that $\Sha(E^d)[2] \ne 2\Sha(E^d)[4]$. Assuming the parity conjecture for $E^d$, the following is true.
	\begin{enumerate}
		\item [a)] If $d<0$ and $d \equiv 1 \pmod{4}$ then $\langle H^d, F_1^{-d}\rangle_{CT}=1$. In particular, $\iota(H^d) \ne 0 \in \Sha(E^d)[2]$. 
		\item [b)] If $d<0$ and $d \equiv 3 \pmod{4}$ then $\iota(H^d) \ne 0$ if and only if $\langle H^d, F_2^{-d}\rangle_{CT}=1$. 
		\item[c)] If $d>0$ then $\iota(H^d) \ne 0$ if and only if $\langle H^d, H_1^d\rangle_{CT}=1$ or $\langle H^d, H_2^d\rangle_{CT}=1$.
	\end{enumerate}
\end{theorem}

It remains to explain how to compute the Cassels-Tate pairing of the quadratic twists of quartics.
To each pair $(A,B)$ of quartics from  Table \ref{table:fields} (see \eqref{eq:quartics}), by the work of Smith (see Theorem 3.2. in \cite{Smith}), we can associate the governing field $L_{A,B}$ such that 
the value of pairing $\CT{A^d}{B^d}$ is determined by $\CT{A}{B}$ and the splitting behaviour of $d$ in $L_{A,B}$. For example, for $d\in T$, it follows that $\CT{H^d}{H^d_2}=0$ if and only if $d$ splits completely in $L=\QQ(\sqrt{-1},\sqrt{2},\sqrt{13})(\sqrt{4+2\sqrt{13}})$. For complete description of governing fields see Table \ref{table:fields} and Section \ref{sec:governing}. Section \ref{sec:governing} and Proposition \ref{prop:sha24} imply the following corollary of Theorem \ref{thm:main}.

\begin{table}[h!]
	\centering
	\begin{tabular}{|c|c|c|}
		\hline
		$\CT{A^d}{B^d}$ &  $K_{A,B}$ & $\alpha_{A,B}$ \\
		\hline
		$\CT{H^d}{H_1^d}$ & $\QQ(\sqrt{3},\sqrt{13})$ &  $4+\sqrt{13}$\\
		$\CT{H^d}{H_2^d}$ & $\QQ(\sqrt{-1},\sqrt{2}, \sqrt{13})$ & $4+2 \sqrt{13}$\\
		$\CT{H^{-d}}{F_1^d}$ & $\QQ(\sqrt{-2},\sqrt{13})$ & $-1$\\
		$\CT{H^{-d}}{F_2^d}$ & $\QQ(\sqrt{13},\sqrt{-1}, \sqrt{-3})$ & $3(1+\sqrt{13})(3+\sqrt{13})$\\
		$\CT{H_1^{d}}{H_2^d}$ & $\QQ(\sqrt{3}, \sqrt{-1}, \sqrt{2})$ & $8(1+\sqrt{3})(4+2\sqrt{3})$\\ 
		$\CT{F_1^{d}}{F_2^d}$ & $\QQ(\sqrt{3}, \sqrt{-1}, \sqrt{2})$ & $8(1+\sqrt{3})(4+2\sqrt{3})$\\
		\hline
	\end{tabular}
	\label{table:fields}
	\vskip 0.1cm
	\caption{ For $d=p > 0$ which splits completely in $K_{A,B}$ (and in the case $\CT{H^d}{H_1^d}$ we in addition require $p \equiv 1 \pmod{4}$), we have $\CT{A^d}{B^d}=0$ if and only if $d$ splits completely in a governing field $L_{A,B}=K_{A,B}(\sqrt{\alpha_{A,B}})$.}
\end{table}

\begin{corollary}\label{cor:main}
	Let $d \in T$. Assuming the parity conjecture for $E^d$, if $d$ does not split completely in $L_{H_1,H_2}=L_{F_1,F_2}$ and
	\begin{enumerate}
		\item [a)] $d=-p<0$ with $p \equiv 1 \bmod{4}$ and $p$ splits completely in $L_{H^{-1},F_2}$, or
		\item [b)] $d=p>0$ and $p$ splits completely in $L_{H,H_1}$ and $L_{H,H_2}$,
	\end{enumerate}
 then $H^{d}(\QQ) \ne \emptyset $. Hence, for such $d$ there exists infinitely many $D(d)$-quintuples.
\end{corollary}

\begin{remark}\label{rem:noT}
As we already noted, if for $d=\pm p$ we have that $\legendre{p}{13}=1$ and $\legendre{p}{2}\cdot\legendre{p}{3}\cdot\legendre{p}{13}=1$ (hence $w(E^d)=-1$ by Proposition \ref{prop:root}), but if  $d \notin T$ (hence $\dim_{\FF_2}\Sel(E^d)=3$), then by Conjecture \ref{conj:2} $\Sha(E^d)[2]$ is trivial, and $H^d(\QQ) \ne \emptyset$, so there exists infinitely many $D(d)$-quintuples.
\end{remark}

\begin{example}
	The set of $d\in T$, $|d| < 3000$, for which Corollary \ref{cor:main} implies that $H^{d}(\QQ) \ne \emptyset$ is equal to
	$$\{  -2857, -2833, -1993, -601, -337, -313,
	1993, 2833, 2857\}.$$
	For $d=-313$, we find a point  $(-2107/1202,389073/1444804)\in H^{-313}(\QQ)$ which produces a $D(-313)$-quintuple
	\footnotesize\begin{equation*}
	\left\{\frac{81062614477261}{1313828969096}, \frac{15660515591}{623554328}, \frac{9009021853}{546517874}, \frac{28246175292437}{1313828969096}, \frac{2532614}{129691}\right\}.
	\end{equation*}\normalsize
\end{example}
\begin{remark}
	Results about infinite number of $D(d)$-quintuples obtained as above from $d\in T$ where $d<0$ are new, they are not covered in \cite{Drazic}.
\end{remark}

Using Chebotarev's density theorem to determine the factorization of primes in governing fields, we obtain the following bounds for $S(X)$.

\begin{corollary}\label{cor:limit}
	Assuming Conjecture \ref{conj:1}, we have that as $X \rightarrow \infty$
	$$C_1 +o(1) \le \frac{\#S(X)}{2\pi(X)} \le C_2  + o(1),$$
	where $C_1=\frac{43}{256}$ and $C_2=\frac{46}{256}$.
\end{corollary}
\begin{remark}
We can rephrase the result above by saying that the classes $H^d\in \Sel(E^d)$, for $d\in T$, are ``equidistributed'' in the quotient $\Sel(E^d)/\kappa({E^d[2]})-\{0\}$ with respect to the image of rational points $E^d(\QQ)$ (rank is generically $1$) under the Kummer map $\kappa:E^d(\QQ)/2E^d(\QQ)\rightarrow \Sel(E^d)$, since we have that the probability for $H^d \in \kappa({E^d(\QQ)})/\kappa({E^d[2]})\subset \Sel(E^d)/\kappa({E^d[2]})$ is $1/7$. Recall that by Proposition \ref{prop:infinity} $H^d$ is never an element of $\kappa({E^d[2]})$ (thus $1/7$ and not $1/8$ is the ``right'' answer).
\end{remark}

By Proposition \ref{prop:sha24} and Conjecture \ref{conj:1} the density of $d$'s ($|d|$ is a prime) for which $\Sha(E^d)[2]$ is nontrivial and $\Sha(E^d)[2] =2\Sha(E^d)[4]$ is $\frac{3}{256}$, so with this method we can not bridge the gap between $C_1$ and $C_2$.

\section{Local properties}

\begin{proposition}\label{prop:ELS}
	For a square-free $d\in \ZZ$, the quartic $H^d$ is everywhere locally solvable if and only if for all primes $p|d$ we have $\legendre{p}{13}=1$ or $p=13$.
\end{proposition}
\begin{proof}
	Assume that $H^d$ is ELS. It follows that for every prime $p|d$, $p\ne 13$, the equation $(x^2-x-3)(x^2+2x-12)=0$ has a solution in $\Fp$, which implies that $\legendre{p}{13}=1$ since the discriminant of quadratic factors is $13$ and $4\cdot 13$ respectively.
	
	Conversely assume that for all primes $p|d$ we have $\legendre{p}{13}=0$ or $1$. Obviously,  $H^d(\RR)\ne \emptyset$. If $p|d$, then by assumption there is a solution $(x^2-x-3)(x^2+2x-12)=0$ in $\Fp$ which lifts by Hensel lemma to $H^d(\QQ_p)$. If $p \nmid 2\cdot 3\cdot 13d$, then $H^d$ has a good mod $p$ reduction since $2,3$ and $13$ are only primes dividing discriminant of $(x^2-x-3)(x^2+2x-12)$. It follows that $H^d/\Fp$ is a genus one curve, hence $H^d(\Fp)\ne \emptyset$, thus by Hensel's lemma $H^d(\QQ_p)\ne \emptyset$. It remains to consider cases $p=2,3,13$ and $p \nmid d$. Here reductions mod $2$, $3$ and $13$ of $H^d$ are geometrically irreducible genus zero curve, so it follows that  $H^d(\Fp)\ne \emptyset$, and consequently $H^d(\QQ_p) \ne \emptyset$ for $p=2,3,13$.
\end{proof}
\begin{remark}
	Novak \cite{Novak_diplomski} showed (assuming GRH) that asymptotically the number of squarefree $d$'s, $0<d<x$, for which $H^d$ is ELS is equal to
	$$\frac{2\sqrt{273}}{13}\pi^{-3/2}\prod_p \left(1+\frac{1}{p} \right)^{\legendre{p}{13}/2}\frac{x}{\sqrt{\log{x}}}.$$
\end{remark}
Similarly, the following proposition describes local solvability of quartics from \eqref{eq:quartics}.
\begin{proposition}\label{prop:ELS_quartics}
	Let $p$ be a prime and $d=\pm p$.
	\begin{enumerate}
		\item [a)] $H_1^d$ is everywhere locally solvable if and only if $d\equiv 1 \pmod{12}$ or $d=-3$.
		\item [b)] $H_2^d$ is everywhere locally solvable if and only if $d>0$ and $d\equiv 1 \pmod{8}$.
		\item [c)] $F_1^d$  is everywhere locally solvable if and only if $d>0$ and $d\equiv 1,3 \pmod{8}$.
		\item [d)] $F_2^d$  is everywhere locally solvable if and only if $d>0$ and $d\equiv 1 \pmod{12}$.
	\end{enumerate}
\end{proposition}

The following proposition computes the root number of $E^d$.

\begin{proposition}\label{prop:root}
	For $d=\pm p$ where $p\ne 2,3,13$ is a prime, the root number $w(E^d)$ is equal to $-1$ if and only if
	$$\legendre{p}{2}\cdot\legendre{p}{3}\cdot\legendre{p}{13}=1.$$
	Here $\legendre{\cdot}{2}$ is the Kronecker symbol for odd $d$ defined by $$\legendre{d}{2}=\begin{cases}
		1, &\textrm{if } |d|\equiv 1,7 \bmod(8)\\
		-1, &\textrm{if } |d|\equiv 3,5 \bmod(8).
	\end{cases}$$
\end{proposition}
\begin{proof}
	Theorem 1.1. in \cite{Desjardins} implies that 
	\begin{equation}\label{eq:root}
	w(E^d)=-w_2(E^d)w_3(E^d)w_{13}(E^d)\legendre{-1}{p},
	\end{equation}
	where $w_p(E^d)$ is a local root number at $p$ of $E^d$. Since $E^d$ has multiplicative reduction at $13$, Proposition 2 in \cite{Rohrlich93} implies that $w_{13}(E^d)=-\legendre{6b}{13}$ where $b= 64108800 d^3$, thus $w_{13}=-\legendre{p}{13}$ (since $w_{13}(E^d)=-1$ if and only if the reduction is split multiplicative). Likewise, for $p\ne 3$, $E^d$ has multiplicative reduction at $3$, hence $w_3(E^d)=-\legendre{d}{3}=-sgn(d)\legendre{p}{3}$. Moreover, since $j$-invariant  $j(E^d)=\frac{22235451328}{123201}$ is integral at $2$, $E^d$ has additive, potentially good reduction at $2$. One can check that $v_2(c_4(E^d))=4$, $c_4(E^d)/2^4 \equiv 3 \pmod{4}$, $v_2(c_6(E^d))=7$ and $v_2(\Delta(E^d))=6$, hence it follows from Table 1. in \cite{Halberstadt} that $w_2(E^d)=1$ if and only if $c_4/2^4-4 c_6/2^7\equiv 7,11 \pmod{16}$. Thus, one can check that $w_2(E^d)=1$ if and only if $d\equiv 1,3 \pmod{8}$, or equivalently $w_2(E^d)=\legendre{-1}{d}\legendre{d}{2}=sgn(d)\legendre{-1}{p}\legendre{p}{2}$. Claim now follows from \eqref{eq:root}.
\end{proof}

\section{Structure of $\Sel(E^d)$}

In this section we describe the structure of $\Sel(E^d)$ in the case when $|d|$ is prime. We prove the following proposition.

\begin{proposition}\label{prop:sha_trivial}
For prime $p\ne 2, 3, 13$, let $d=\pm p$ be such that $\legendre{d}{13}=1$ and $w(E^d)=-1$.
\begin{enumerate}
	\item [a)] If $d\in T$ (i.e. $d\equiv 1 \pmod{8}$ if $d>0$ or $d\equiv 5,7 \pmod{8}$ if $d<0$), then $\dim_{\FF_2}\Sel(E^d)=5$. More precisely,  if $d>0$, then $\Sel(E^d)$  is generated by torsion classes, $H^d$, $H_1^d$ and $H_2^d$. If $d<0$, then $\Sel(E^d)$ is generated by torsion classes, $H^d$, $F_1^{-d}$, and $F_2^{-d}$ if $d\equiv 7 \pmod{8}$ or $H_1^d$ if $d\equiv 5 \pmod{8}$. 
	\item[b)] If $d \notin T$, then we have that $\dim_{\FF_2}\Sel(E^d)=3$.
\end{enumerate}
\end{proposition}

Since $E$ has full $2$-torsion over $\QQ$, each class in $H^1(\QQ,E[2])$ can be identified with an element of $(\QQ^\times/\QQ^{\times 2})^3$ in the following way. Denote by $P_1=(8,0),P_2=(-18,0)$ and $P_3=(9,0)$ nontrivial elements in $E[2]$, by $e_2:E[2]\times E[2] \rightarrow \mu_2$ the Weil pairing (hence $e_2(P_i,P_j)=-1$ if and only if $i\ne j$), and by $\omega:E[2] \rightarrow \Hom(E[2],\mu_2^3)$, $T \mapsto (P_i \mapsto e_2(T,P_i))$ the group homomorphism induces by $e_2$. For each class $F \in H^1(\QQ,E[2])$, we denote by $\omega_*(F)$ the pushforward of $\omega$ from $H^1(\QQ,E[2])$ to $H^1(\QQ,\mu_2^3)\cong H^1(\QQ,\mu_2)^3\cong(\QQ^\times/\QQ^{\times 2})^3$ where the last isomorphism is given by the Kummer map sending $\alpha \in \QQ^\times/\QQ^{\times 2}$ to $\xi \in H^1(\QQ,\mu_2)$ such that $\xi(\sigma) = \frac{\sqrt{\alpha}^\sigma}{\sqrt{\alpha}}$ for every $\sigma \in \GalQ$. One has that $\omega_*(F)=(a_1,a_2,a_3)$ is equivalent to $F(\sigma)=\chi_{a_1}(\sigma)P_1+\chi_{a_2}(\sigma)P_2$, for all $\sigma \in \GalQ$, and $a_1 a_2 a_3 \in \QQ^{\times 2}$, where, for $a\in \QQ$. Here, we denote by $\chi_a$ the nontrivial character of $\QQ(\sqrt{a}))$ with values in $\ZZ/2\ZZ$ (if $\QQ(\sqrt{a})=\QQ$ then $\chi_{a}$ is trivial). It follows that $F$ is defined over $\QQ(\sqrt{a_1}, \sqrt{a_2})$.

We start with the following standard lemma.

\begin{lemma} \label{lem:twisted_classes_w}
	For elliptic curve $\tilde{E}:y^2=(x-a_1)(x-a_2)(x-a_3)$, where $a_1,a_2,a_3\in \QQ$, let $F$ be a quartic $y^2=g(x)$, $g(x)\in \ZZ[x]$, isomorphic (over $\overline{\QQ}$) to $\tilde{E}$,  which represents an element in $H^1(\QQ,\tilde{E}[2])$ (the quartic is not necessarily everywhere solvable, i.e. the element of the $2$-Selmer group). For $d\in \ZZ$ let $F^d$ be the quadratic twist of $F$, thus representing the element in $H^1(\QQ,\tilde{E}^d[2])$. After identifying $H^1(\QQ,\tilde{E}[2]) \cong H^1(\QQ,\tilde{E}^d[2])$, we have
	$$\omega_*(F)=\omega_*(F^d).$$
\end{lemma}
\begin{proof}
	The claim follows directly from the interpretation of the map $\omega_*$  in terms of two-descent theory. If $\omega_*(F)=(q_1,q_2,q_3)$, then $F$ is isomorphic (over $\QQ$) to
	the curve 
	\begin{align*}
		q_1 y^2&=x-a_1,\\
		q_2 y^2&=x-a_2,\\
		q_3 y^2&=x-a_3,
	\end{align*}
	while its twist over $\QQ(\sqrt{d})$ is given by
$$ q_1 y^2=x-d a_1, \quad q_2 y^2 =x-d a_2, \quad q_3 y^2=x-d a_3.$$
where the isomorphism $F\rightarrow F^d$ maps $(x,y_1,y_2,y_3)\mapsto (dx,\sqrt{d}y_1,\sqrt{d}y_2,\sqrt{d}y_3)$.
Since $\tilde{E}^d$ is isomorphic to $y^2=(x-d a_1)(x-d a_2)(x-d a_3)$, we recognize from the above that $\omega_*(F^d)=(q_1,q_2,q_3)$ (we identified $(a_i,0)$ with $(d a_i,0)$), and the claim follows.
\end{proof}
For the proof of Proposition \ref{prop:sha_trivial}, we need to introduce three more quartics.
\begin{align*}
		H_3&: y^2 = 25x^4 + 48x^3 - 114x^2 - 144x + 225\in Sel(E^3),\\
		F_3&: y^2=-71x^4 - 336x^3 - 538x^2 - 336x - 71 \in \Sel(E^{-1}),\\
		F_4&: y^2=-5x^4 + 76x^3 - 168x^2 - 296x - 92 \in \Sel(E^{-3}).
\end{align*}

Recall that quadratic twist $E^d$ has the  Weierstrass model $E^d:y^2=(x-8d)(x-9d)(x+18d)$.

Next, we prove linear independence of classes needed for the proof of Proposition \ref{prop:sha_trivial}.

\begin{lemma} \label{lem:independence} For $d\in \ZZ$, $|d|$ prime, and $|d|\notin \{2,3,13\}$, denote by $Q_1=(8d,0)$ and $Q_2=(-18d,0)$ elements in $E^d[2]$ which correspond to $P_1$ and $P_2$ under the natural isomorphism $E[2] \cong E^d[2]$, and by $\kappa:E^d(\QQ)/2E^d(\QQ)\rightarrow \Sel(E^d)\subset H^1(\QQ,E^d[2])$ the Kummer map. We have that
\begin{align*}
	\omega_*(H^d)&=(13,13,1), \omega_*(\kappa(Q_1))=(26d,-26,-d),\omega_*(\kappa(Q_2))=(78,-26d,-3d),\\
	\omega_*(H_1^d)&=(3,1,3), \omega_*(H_2^d)=(2,-2,-1),\omega_*(H_3^d)=(6,-6,1),\\
	\omega_*(F_1^d)&=(-2,-2,1),\omega_*(F_2^d)=(-3,-1,3), \omega_*(F_3^d)=(6,2,3),\omega_*(F_4^d)=(6,6,1).
\end{align*}
Moreover,
	\begin{enumerate}
		\item [a)] if $d>0$, the the classes $\omega_*(F)$, for $F \in \{\kappa(Q_1),\kappa(Q_2),H^d,H_1^d,H_2^d,H_3^d,F_3^{-d}\}$ are (multiplicatively) independent in $(\QQ^\times/\QQ^{\times 2})^3$ and locally solvable at infinity,
		\item [b)] if $d<0$, the classes $\omega_*(F)$, for $F \in \{\kappa(Q_1),\kappa(Q_2),H^d,H_1^{d},F_1^{-d},F_3^{-d},F_4^{-3d}\}$ are (multiplicatively) independent in $(\QQ^\times/\QQ^{\times 2})^3$, and locally solvable at infinity.
	\end{enumerate}
\end{lemma}
\begin{proof}
Using Magma \cite{Magma}, we can easily compute the values of $\omega_*(F^d)$ for quartics $F$ from \ref{eq:quartics} as they don't depend on $d$ by Lemma \ref{lem:twisted_classes_w}.

We can also compute classes of torsion points explicitly. For example, for $Q_1=(8d,0)\in E^d(\QQ)$, one can check that $2R_1=Q_1$, where $R_1=(\frac{1}{2}r^2-\frac{9d}{2},\frac{1}{2}r^3-\frac{25d}{2})$, with $r^4-50d r^2+3^6 d^2=0$. Here $\QQ(r)=\QQ(\sqrt{-d},\sqrt{-26})$, and by inspection
one obtains that $R_1^\sigma -R_1=\chi_{26d}(\sigma)Q_1+\chi_{-26}(\sigma)Q_2$, thus $\omega_*(\kappa(Q_1))=(26d,-26,-d)$. Similarly, one computes $\omega_*(\kappa(Q_2))$. 

 The existence of real points on quartic (which determine local solvability at infinity) can be checked for each quartic separately.  

If $d>0$, it is not hard to see that the classes will be independent unless $d$ is divisible only by $2,3$ and $13$.
In particular, for squarefree $d$, we compute that this happens for $\{1, 2, 3, 6, 13, 26, 39, 78\}$, thus the claim in $a)$ follows. The claim in b) is proved in a similar way.
	
\end{proof}

We have the following proposition as a consequence of the previous lemma.

\begin{proposition}\label{prop:infinity}
	If $d\in \ZZ$ is square free integer such that $H^d(\QQ) \ne \emptyset$, then $H^d(\QQ)$ is infinite.
\end{proposition}
\begin{proof}
	Assume that for some $d\in \ZZ$, $H^d(\QQ) \ne \emptyset$ and  $H^d(\QQ)$ is finite.
	It follows that the rank of Mordell-Weil group of $E^d(\QQ)$ is zero, hence $H^d$ as an element of $2$-Selmer group $\Sel(E^d)$ is in the image of the two torsion $E^d[2]$ under the map $E^d(\QQ)/2E^d(\QQ) \hookrightarrow \Sel(E^d)$ from \eqref{eq:exact}. More precisely, there is a point of order $4$, $Q\in E^d[4]$, such that $H^d$ corresponds to the cocycle $\sigma \mapsto Q^\sigma-Q$. It follows from Lemma \ref{lem:independence} that the image of this cocycle is 
	of order $2$ which implies that $Q$ is defined over quadratic field. There are only finitely many $d's$ that have a point of order $4$ defined over quadratic field.
	Note that if $x_0$ is an $x$-coordinate of point of order $4$ on $E$ (it is defined over quadratic field), then $d\cdot x_0$ is an $x$-coordinate of point of order $4$ on $E^d$. Moreover, if $E_d:y^2=f_d(x)=(x-8d)(x-9d)(x+18d)$, then $f_d(d x_0)=d^3\cdot f_1(x_0)$ is a square in $\QQ(x_0)$ if and only if $d \cdot f_1(x_0)$ is a square. One can check that this is the case if and only if $d=\{-26,-3,-1,1,3,26\}$. The proposition follows after verifying the claim for these special cases.
\end{proof}

To obtain an upper bound for the size of $2$-Selmer group, we will use the method and terminology from the paper of Mazur and Rubin \cite{Mazur_Rubin_Hilbert}[Section 3] (see also \cite{Kramer, Boxer_Diao}).

\begin{definition}
	Suppose $\tilde{E}$ is an elliptic curve over $\QQ$. For every place $v$ of $\QQ$, let $H_f(\QQ_v,\tilde{E}[2])$ denote the image of the Kummer map
	$$\tilde{E}(\QQ_v)/2\tilde{E}(\QQ_v) \rightarrow H^1(\QQ_v,\tilde{E}[2]).$$
	The $2$-Selmer group $\Sel(\tilde{E})$ is the $\FF_2$-vector space defined by the exactness of the sequence
	$$0\rightarrow \Sel(\tilde{E})\rightarrow H^1(\QQ,\tilde{E}[2])\rightarrow \bigoplus_v H^1(\QQ_v,\tilde{E}[2])/H^1_f(\QQ_v,\tilde{E}[2]).$$
	We say that $2$-Selmer group $\Sel(\tilde{E})$ is cut out by the local conditions $H_f(\QQ_v,\tilde{E}[2])$.
\end{definition}
The following lemma describes the size of local conditions.

\begin{lemma}\label{lem:size_Hf}
Let $v$ be a finite rational place and $d$ an odd squarefree integer. We have
\[
\dim_{\FF_2}H_f^1(\QQ_v,E^d[2]) = \begin{cases}
	2& \textrm{ if } v \ne 2\\
	3& \textrm{ if } v=2.\\
\end{cases}
\]
\end{lemma}

\begin{proof}
By Lemma 2.2 in \cite{Mazur_Rubin_Hilbert}, if $v  \nmid \, 2\infty$, then $\dim_{\FF_2}H_f^1(\QQ_v,E^d[2])=\dim_{\FF_2}E^d(\QQ_v)[2]=2$.

Following \cite[Chapter 4.]{Silverman_Arithmetic}, denote by $\mathcal{F}$ the formal group associated to the elliptic curve $E^d/\QQ_2$, and by $\mathcal{F}(2\ZZ_2)$ the group associated to that formal group. Theorem 6.4. b) in \cite{Silverman_Arithmetic} implies that $\mathcal{F}(4\ZZ_2)$ is isomorphic (via formal logarithm map) to the additive group $\hat{\mathbb{G}}_a(4\ZZ_2)$ which implies that $\mathcal{F}(4\ZZ_2)/2\mathcal{F}(4\ZZ_2)\cong \ZZ/2\ZZ$. On the other hand, since $\mathcal{F}(x,y)=x+y-a_1 xy-a_2(x^2 y+x y^2)+\cdots$, where $a_1$ and $a_2$ are the usual Weierstrass coefficients of $E^d$, it follows that $[2](x)=2x+O(x^3)$ (as $a_1=0$), thus $2\mathcal{F}(2\ZZ_2)=\mathcal{F}(4\ZZ_2)$. In particular, $\mathcal{F}(2\ZZ_2)/2\mathcal{F}(2\ZZ_2)\cong \ZZ/2\ZZ$.

If we denote by $E^d_1(\QQ_2)$ the subgroup of points in $E^d(\QQ_2)$ which reduce to the point at infinity modulo two, then it is well known that $E^d_1(\QQ_2) \cong \mathcal{F}(2\ZZ_2)$. Moreover, $E^d_0(\QQ_2)/E^d_1(\QQ_2)$, where $E^d_0(\QQ_2)$ is the subgroup of points of nonsingular reduction, is generated by two torsion point with odd $x$ coordinate. Finally, $E^d(\QQ_2)/E^d_0(\QQ_2)$ is generated by the point of order two with even $x$ coordinate (Tamagawa number of $E^d$ is two), and we have that $E^d(\QQ_2)/2E^d(\QQ_2)\cong (\ZZ/2\ZZ)^3$, so the claim follows.
\end{proof}

There is a natural identification of Galois modules $E[2]\cong E^d[2]$ - which is crucial for our argument. We identify point $(a,0)\in E(\QQ)$ with $(8a,0)\in E^d(\QQ)$ for $a\in \{8,9,-18\}$. It allows us to view $\Sel(E^d)$ as a subspace of the $H^1(\QQ,E[2])$, but defined by the different sets of local conditions $H^1_f(\QQ_v,E^d[2])\subset H^1(\QQ_v,E[2])$.

\begin{definition}
	If $\T$ is a finite set of places of $\QQ$, define relaxed $2$-Selmer group $\mathcal{S}^{\T}$ by the exactness of 
	$$0\rightarrow \mathcal{S}^{\T} \rightarrow H^1(\QQ,E[2]) \rightarrow \bigoplus_{v \notin \T} H^1(\QQ_v,E[2])/H^1_f(\QQ_v,E[2]),$$
	where the second arrow is induced by the sum of localization maps $H^1(\QQ,E[2]) \rightarrow H^1(\QQ_v,E[2])$.
\end{definition}

By definition $\Sel(E) \subset \mathcal{S}^{\T}$ for any $\T$. We will choose $\T$ such that $\Sel(E^d)\subset \mathcal{S}^{\T}$ holds as well.
For that we will need the following criteria for equality of local conditions after twist (see Lemma 2.10 and Lemma 2.11 in \cite{Mazur_Rubin_Hilbert}).

\begin{lemma}\label{lem:local_conditions}
	Let $\tilde{E}/\QQ$ be an elliptic curve. Let $v$ be a place of $\QQ$ and $d$ a squarefree integer. If at least one of the following conditions holds
	\begin{enumerate}
		\item [a)] $v$ splits in $\QQ(\sqrt{d})$,
		\item [b)] $v$ is a prime of good reduction of $\tilde{E}$ and $v$ is unramified in  $\QQ(\sqrt{d})/\QQ$,
	\end{enumerate}
then $H^1_f(\QQ_v,\tilde{E}[2])=H^1_f(\QQ_v,\tilde{E}^d[2])$.
	Moreover, if $\tilde{E}$ has good reduction at $v$, and $v$ is ramified in $\QQ(\sqrt{d})/\QQ$, then
	$$H^1_f(\QQ_v,\tilde{E}[2])\cap H^1_f(\QQ_v,\tilde{E}^d[2])=0.$$
\end{lemma}

Since primes of bad reduction of $E^d$ are $\{2,3,13,p\}$, and since $13$ splits in $\QQ(\sqrt{d})$, it follows from Lemma \ref{lem:local_conditions} that local conditions $H^1_f(\QQ_v,E^d[2])$ and $H^1_f(\QQ_v,E[2])$ are equal outside the set $\T=\{2,3,p,\infty\}$.

\begin{proof}[Proof of Proposition \ref{prop:sha_trivial}]
Lower bound for the $\dim_{\FF_2}\Sel(E^d)$ in both cases follows from Lemma \ref{lem:independence} and Proposition \ref{prop:ELS_quartics}. Note that if $d \in T$, the classes $H_1^d$ and $H_2^d$ are ELS if $d>0$, classes $F_1^{-d}$ and $F_2^{-d}$ are ELS if $d<0$ and $d \equiv 7 \pmod{8}$, and classes  $F_1^{-d}$ and $H_1^d$ are ELS if $d<0$ and $d \equiv 5 \pmod{8}$.

For the upper bound we first consider the case $d \in T$. From the definition of $T$ it follows that for $|d|>3$ primes $2$ and $3$ split in $\QQ(\sqrt{d})$, thus Lemma \ref{lem:local_conditions} implies that local conditions  $H^1_f(\QQ_v,E^d[2])$ and $H^1_f(\QQ_v,E[2])$ differ only at $v=p$ (and possibly at $v=\infty$ if $d<0$ - note that if $d>0$ elliptic curves $E$ and $E^d$ are isomorphic over $\RR$). 

Assume that $d>0$ and set $\T=\{p\}$. Define a strict $2$-Selmer group $\mathcal{S}_{\T}:=\mathcal{S}_{\T}(E)$ by the exactness of 
	$$0\rightarrow \mathcal{S}_{\T} \rightarrow \mathcal{S}^{\T} \rightarrow \bigoplus_{v \in \T} H^1(\QQ_v,E[2]),$$
where the second arrow is the sum of the localization maps. 

From the construction, it follows that $\mathcal{S}_{\T}\subset \Sel(E^d)\subset\mathcal{S}^{\T}$, and  $\mathcal{S}_{\T}\subset \Sel(E) \subset \mathcal{S}^{\T}$. We will show that $\mathcal{S}_{\T} = \Sel(E)$. One can compute that
$E(\QQ)$ is generated by $2$-torsion points $S_1=(-18,0)$, $S_2=(8,0)$ and point $S_3=(45/4,-117/8)$ of infinite order, and that $\Sel(E)$ is generated by the $\kappa(S_i)$, $i=1,2,3$,  where $\kappa:E(\QQ)/2E(\QQ)\rightarrow \Sel(E)$ is the Kummer map - thus $\dim_{\FF_2}\Sel(E)=3$. It is enough to show that the image of the $\kappa(S_i)$ in $H^1(\QQ_p,E[2])$ is trivial. Choose $Q_i \in E(\overline{\QQ})$ such that $2 Q_i = S_i$. The fields of definitions $K_i$ of points $Q_i$ are $K_1=\QQ(\alpha_1)$ where  $\alpha_1^4+106 \alpha_1^2+1=0$, $K_2=\QQ(\alpha_2)$ where $\alpha_2^4-50 \alpha_2^2+729=0$ and $K_3=\QQ(\alpha_3)$ where $\alpha_3^2-3\alpha_3-43/4=0$. It happens that $p$ splits completely in all the fields, hence the claim follows. 

Lemma 3.2 in \cite{Mazur_Rubin_Hilbert} implies that $\dim_{\FF_2}\mathcal{S}^{\T}-\dim_{\FF_2}\mathcal{S}_{\T}=\dim_{\FF_2}H^1_f(\QQ_p,E[2])$. By Lemma \ref{lem:size_Hf} and inclusion $\Sel(E^d)\subset \mathcal{S}^T$, it follows $\dim_{\FF_2}\Sel(E^d)\le 3+2=5$, and the claim follows.

 The case $d<0$ is analogous - to get the equality of local conditions at $v=\infty$ one replaces $E$ with $E^{-1}$, and then proceeds as in the $d>0$ case.

Now assume that $d\notin T$. Consider the case $d<0$. In the case $d>0$ one repeats the same argument with $E^{-1}$ replaced by $E$. Primes $2$ and $3$ do not need to split in $\QQ(\sqrt{d})$ any more, hence we set $\tilde{T}=\{2,3,p\}$ and  $\mathcal{S}^{\T}:=\mathcal{S}^{\T}(E^{-1})$ and $\mathcal{S}_{\T}:=\mathcal{S}_{\T}(E^{-1})$ (we replaced $E$ with $E^{-1}$ in definitions to ensure the equality of local conditions at $v=\infty$). Lemma 3.2 in \cite{Mazur_Rubin_Hilbert} and Lemma \ref{lem:size_Hf} imply that $\dim_{\FF_2}\mathcal{S}^{\T}-\dim_{\FF_2}\mathcal{S}_{\T}=\dim_{\FF_2}H^1_f(\QQ_2,E^{-1}[2])+\dim_{\FF_2}H^1_f(\QQ_3,E^{-1}[2])+\dim_{\FF_2}H^1_f(\QQ_p,E^{-1}[2])=3+2+2= 7$. Since $\mathcal{S}_{\T}\subset \Sel(E^{-1})$, if we show that the image of each class in $\Sel(E^{-1})$ (which is generated by $H^{-1}, F_1$ and $F_3$) under the localization $\loc_2:\Sel(E^{-1}) \rightarrow H^1(\QQ_2,E^{-1}[2])$ is different than zero, then it follows that $\mathcal{S}_{\T}={0}$. One can check that, for any $P\in E^{-1}(\QQ)/2E^{-1}(\QQ)$ and $Q\in E^{-1}(\overline{\QQ})$ such that $2 Q=P$, $2$ is ramified in the field of definition of $Q$, hence the localization of $\kappa(P)$ at $v=2$ is nontrivial, and $\mathcal{S}_{\T}={0}$. It follows that $\dim_{\FF_2}\mathcal{S}^{\T}=7$. 

Lemma \ref{lem:independence} b) provides us with the generators of $\mathcal{S}^{\T}$ once we show that the torsion classes together with classes $H,F_1,H_1,F_2,F_4 \in H^1(\QQ,E)$ satisfy local conditions $H^1_f(\QQ_v,E^{-1}[2])$ for $v$ outside the set $\T$. Equivalently, one can check that the quartics $H^{-1},F_1,H_1^{-1},F_2$ and $F_4^{3}$ (as two covers of $E^{-1}$) are locally solvable outside the set $\T$. Local solvability at the finite places outside the set $\{2,3,13\}$ follows immediately from Hensel lemma argument (as in the proof of Proposition \ref{prop:ELS}) since these are the bad primes of $E^{-1}$, while solvability at $v=\infty$ (i.e. existence of the real points on quadratic twists) follows from the observation that polynomials of degree $4$ defining $H$ and $H_1$ have real roots. The local solvability at $v=13$ follows from the fact that $\legendre{p}{13}=1$, which implies that $p$ is a square in $\QQ_{13}$, thus quadratic twist by $\QQ(\sqrt{d})$ or $\QQ(\sqrt{-d})$ of any quartic from Lemma \ref{lem:independence} b) is  isomorphic over $\QQ_{13}$ to that quartic. Hence, we only need to check that $F_4^3$ is locally solvable at $v=13$ which is checked readily.

We will prove that $\dim_{\FF_2} \Sel(E^d) \le 4$, which will imply that $\dim_{\FF_2} \Sel(E^d)=3$ since $\dim_{\FF_2} \Sel(E^d)$ is odd (by \cite{Dokchitser_Dokchitser} $(-1)^{\dim_{\FF_2}\Sel(E^d)}= w(E^d)=-1$) and greater or equal to $3$ (since $H^d$ and the torsion classes of $E^d$ are linearly independent in $\Sel(E^d)$). Essentially, for each class in $\mathcal{S}^{\T}$ (generators are given by Lemma \ref{lem:independence} b)), we will check if it satisfies the local conditions $H_f^1(\QQ_v,E^{d}[2])$.

Observe that the local condition at $v=p$, $H_f^1(\QQ_p,E^d[2])$, for $p\ne \{2,3,13\}$ is determined with the image of $2$-torsion $\kappa(P_1)(\sigma)=\chi_{3}(\sigma)P_1+\chi_{d}(\sigma)P_3$, $\kappa(P_2)=\chi_{-13d}(\sigma)P_1+\chi_{-2}(\sigma)P_3$ (since the elements are independent and dimension of the local condition is $2$). As the remaining generators of $\mathcal{S}^{\T}$, $H:\sigma \mapsto \chi_{13}(\sigma)P_3$, $F_1:\sigma \mapsto \chi_{-2}(\sigma)P_3$, $H_1:\sigma\mapsto \chi_3(\sigma)P_1$, $F_4:\sigma \mapsto \chi_6(\sigma)P_3$ and $F_2+H_1=\chi_{-1}(\sigma)P_3$ do not depend on $d$ (here $\chi_q$ denotes the nontrivial character of $\QQ(\sqrt{q}$)), the local condition at $v=p$ can be satisfied by some class from the subspace generated by $H,H_1,F_1,F_4$ and $F_2$ only if the localization of that class at $v=p$ is trivial. If $p\equiv 5 \pmod{8}$, then $-1,13$ are squares in $\QQ_p$ while $2$ and $3$ are not, thus $H$, $F_2+H_1$ and $F_4$ generate the subspace of $\mathcal{S}^{\T}$ with required property, while if $p\equiv 7 \pmod{8}$, then $13,2$ are squares in $\QQ_p$ while $-1$ and $3$ are not, thus $H, F_1+F_4$ and $F_1+F_2+H_1$ generate the subspace of $\mathcal{S}^{\T}$ consisting of elements whose localization at $v=p$ is trivial.

Next, to rule out remaining classes, we focus on the local condition at $v=3$. If $p\equiv 5 \pmod{8}$, then $d$ is a square in $\QQ_3$, and the classes $\loc_3\kappa(P_1)(\sigma)=\chi_3(\sigma)P_1$ and $\loc_3\kappa(P_2)(\sigma)=\chi_{-1}(\sigma)P_1$ linearly independent, thus they generate $2$-dimensional $\FF_2$-vector space $H_f^1(\QQ_3,E^d[2])$. Since, $\loc_3(F_4)(\sigma)=\chi_6(\sigma)P_3=\chi_{-3}(\sigma)P_3 \notin H_f^1(\QQ_3,E^d[2])$,  we conclude that in this case $\dim_{\FF_2}\Sel(E^d)\le 4$, hence equal to $3$.
	
If $p\equiv 7 \pmod{8}$, then $\loc_3\kappa(P_1)(\sigma)=\chi_3(\sigma)P_1+\chi_{-1}(\sigma)P_3$ and $\loc_3\kappa(P_2)(\sigma)=0$ generate a $1$-dimensional subspace of the $2$-dimensional vector space $H_f^1(\QQ_3,E^d[2])$.  Note that not all the localisations of the classes of interest $\loc_3(F_1+F_4)(\sigma)=\chi_{-3}(\sigma)P_3$ and $\loc_3(F_1+F_2+H_1)(\sigma)=\chi_{-1}(\sigma)P_3$ can lie in $H_f^1(\QQ_3,E^d[2])$ (since the subspace they generated does not contain $\loc_3\kappa(P_1)(\sigma)$), hence $\dim_{\FF_2}\Sel(E^d)\le 4$, and the claim follows. 
\end{proof}

The following proposition follows immediately from the explicit description of $\Sel(E^d)$ given in Proposition \ref{prop:sha_trivial}. 

\begin{proposition}\label{prop:sha24}
	Let $d\in T$ (hence $\dim_{\FF_2} \Sel(E^d)=5$). We have that $\Sha(E^d)[2]=2\Sha(E^d)[4]$ if and only if 
	\begin{enumerate}
		\item [a)] $\CT{H_1^d}{H_2^d}=0$ and $\CT{H^d}{H_i^d}=0$ for $i=1,2$ if $d>0$,
		\item [b)]  $\CT{F_1^{-d}}{F_2^{-d}}=0$ and $\CT{H^d}{F_i^{-d}}=0$ for $i=1,2$ if $d<0$ and $d \equiv 7 \pmod{8}$,
	\end{enumerate}
\end{proposition}
\begin{proof}
	If $\Sha(E^d)[2]=2\Sha(E^d)[4]$, then the Cassels-Tate pairing on $\Sel(E^d)$ is trivial (since it is non-degenerate on $\Sha(E^d)[2]/2\Sha(E^d)[4]$), hence the claim follows.
	Similarly, if a),b) o holds, then Proposition \ref{prop:sha_trivial} implies the Cassels-Tate pairing on $\Sel(E^d)$ is trivial, hence $\Sha(E^d)[2]=2\Sha(E^d)[4]$.
	Note that in the case $d<0$ and $d \equiv 5 \pmod{8}$, we always have $\CT{H^d}{F_1^{-d}}=1$ (see Theorem \ref{thm:main}a)), hence $\Sha(E^d)[2] \ne 2\Sha(E^d)[4]$.
\end{proof}

\section{Cassels-Tate pairing and governing fields} \label{sec:governing}

Our main tool for studying Cassels-Tate pairing of quadratic twists of elements of $2$-Selmer groups is the following specialisation of the theorem of Smith (see Section $3$ in \cite{Smith}).

\begin{theorem}[Smith]\label{thm:Smith}
	Let $\tilde{E}$ be an elliptic curve over $\QQ$ with full $2$-torsion over $\QQ$. Let
	$$F,F' \in H^1(\QQ,\tilde{E}[2]),$$
	and let $K$ be the minimal field over which $F$ and $F'$ are trivial.
	Next, let $S$ be any set of places of $\QQ$ which contains all places of bad reduction of $\tilde{E}$, the archimedean place and $2$. Take $\mathcal{D}$ to be the set of pairs $(d_1,d_2)$ of elements in $\QQ^\times$ such that $d_1/d_2$ is square at all places of $S$, and $F^{d_1}$ and $F'^{d_2}$ are elements of $2$-Selmer group of $\tilde{E}^{d_1}$ and $\tilde{E}^{d_2}$ respectively. 
	
	If $F\cup F'$ is alternating (as defined in Section $3$ of \cite{Smith}), then $\CT{F^{d_1}}{F'^{d_1}}=\CT{F^{d_2}}{F'^{d_2}}$ for all $(d_1,d_2) \in \mathcal{D}$. Otherwise, there is a quadratic extension $L$ of $K$ that is ramified only at primes in $S$ such that
	$$\CT{F^{d_1}}{F'^{d_1}}=\CT{F^{d_2}}{F'^{d_2}}+\left[\frac{L/K}{\mathbf{d}}\right],$$
	for all $(d_1,d_2)\in \mathcal{D}$, where the Galois group $\Gal(L/K)$ is identified with $\frac{1}{2}\ZZ/\ZZ$. Here $\mathbf{d}$ is any ideal of $K$ coprime to the conductor of $L/K$ that has norm in $\QQ^\times/{\QQ^\times}^2$ equal to $(d_1/d_2)$. Such $\mathbf{d}$ exists for all $(d_1,d_2)\in \mathcal{D}$. We denote by $\left[\frac{\ \cdot\ }{\cdot}\right]$ the Artin symbol.
\end{theorem}
\begin{remark}
	We will call field $L$ from the statement of Theorem \ref{thm:Smith} a governing field of $F$ and $F'$. It needs not to be unique.
\end{remark}

Next, we compute the governing fields of some pairs of classes defined by quartics from \eqref{eq:quartics} (see Table \ref{table:fields}). 

In general, following Section 3.1. in \cite{Smith}, for $F,F'\in H^1(\QQ,E[2])$ let $\omega_*(F)=(a_1,a_2,a_3)$ and $\omega_*(F')=(a_1',a_2',a_3')$.  For every place $v$ we have the following relation of Hilbert symbols $(a_1,a_1')_v (a_2,a_2')_v (a_3,a_3')_v=1$. We can choose $b \in \QQ^\times$ such that  $(a_1,b a_1')_v= (a_2,b a_2')_v = (a_3,b a_3')_v=1$ which implies that we can find $x_i, y_i,z_i\in \QQ^\times$
such that $x_i^2-a_i y_i^2=b a_i' z_i^2$ for $i=1,2,3$. We can further scale $x_i, y_i$ and $z_i$ by a common factor so that the field
$$L_{F_,F'}=K_{F,F'}\left(\sqrt{(x_1+y_1\sqrt{a_1})(x_2+y_2\sqrt{a_2})(x_3+y_3\sqrt{a_3})}\right)$$ avoids ramification at places unramified in the common field of definition $$K_{F,F'}:=\QQ(\sqrt{a_1},\sqrt{a_2},\sqrt{a_1'},\sqrt{a_2'}).$$

\begin{lemma}[Smith]
	If $F \cup F'$ is not alternating and $\deg K_{F,F'}/\QQ = 16$, then $L_{F,F'}$ is a governing field of $F$ and $F'$.
\end{lemma}

Although in our case $\deg K_{F,F'}/\QQ$ is either four or eight, we can still compute governing fields using the following lemma which follows from the proof of Proposition 2.1. in \cite{Smith}.

\begin{lemma}\label{lem:D8}
	For integers $a$ and $b$ such that $ab$ is not a perfect square let $L_{a,b}/\QQ(\sqrt{a},\sqrt{b})$ be quadratic extension such that $L_{a,b}/\QQ$ is Galois with Galois group isomorphic to dihedral group $D_8$. There exist a map $$\gamma_{a,b}:\GalQ \xtwoheadrightarrow{res} \Gal(L_{a,b}/\QQ) \rightarrow \mu_2$$
	which satisfies $d\gamma_{a,b}=\chi_a \cup \chi_b \in H^2(\GalQ,\mu_2)$. Here $\mu_2=\{\pm 1\}$ and the cup product $\chi_a \cup \chi_b$ is induced by the natural bilinear map $\ZZ/2\ZZ \times \ZZ/2\ZZ \rightarrow \ZZ/2\ZZ$ (hence for $\sigma, \tau \in \GalQ$ we have that $(\chi_a\cup \chi_b)(\sigma,\tau)=-1$ if and only if $\sqrt{a}^{\,\sigma} = -\sqrt{a}$ and $\sqrt{b}^{\,\tau} = -\sqrt{b}$).
\end{lemma}

\subsection{ $L_{H^{-1},F_2}= \QQ(\sqrt{13},\sqrt{-1}, \sqrt{-3})(\sqrt{3(1+\sqrt{13})(3+\sqrt{13})})$} \label{sub:jedan}
It follows from Lemma \ref{lem:independence} that $H^{-1}(\sigma)=\chi_{13}(\sigma)P_1+\chi_{13}(\sigma)P_2$ and $F_2(\sigma)=\chi_{-3}(\sigma)P_1+\chi_{-1}(\sigma)P_2$ for all $\sigma\in \GalQ$. If we define the cup product $\cup: H^1(\GalQ,E[2])\times H^1(\GalQ,E[2]) \rightarrow H^2(\GalQ,\mu_2)$ using the Weil pairing $e_2:E[2]\times E[2] \rightarrow \mu_2$, it follows that $H^{-1}\cup F_2=\chi_{13}\cup \chi_{-1} \cdot \chi_{13}\cup\chi_{-3}=\chi_{13}\cup \chi_{3}$. The field $L_{H,F_2}$ has a property that it contains subfield $L/\QQ(\sqrt{13},\sqrt{3})$ such that $L/\QQ$ is $D_8$ extension. Lemma \ref{lem:D8} implies that there exists a map $\Gamma:\GalQ\rightarrow \mu_2$ defined over $L_{H,F_2}$ such that $d\Gamma=\chi_{13}\cup \chi_{3}=H^{-1}\cup F_2$. One can check that $L_{H,F_2}/\QQ$ is unramified outside the set $\{2,3,13\}$ of primes of bad reduction of $E$, hence it follows from the proof of Theorem 3.2. in \cite{Smith} that $L_{H,F_2}$ is governing field of $H^{-1}$ and $F_2$. The choice of field $L_{H^{-1},F_2}$ is particularly nice since it is easy to check that for prime $p$ the Cassels-Tate pairing $\CT{H^{-p}}{F_2^p}$ is equal to $0$ if and only if $p$ splits completely in $L_{H^{-1},F_2}$ provided that $H^{-p}$ and $F_2^p$ define an element in $\Sel(E^{-p})$. It follows from Proposition \ref{prop:ELS_quartics} that $H^{-p}$ and $F_2^p$ are ELS if and only if $p=13$ or $p$ splits completely in the field of definition $K_{H^{-1},F_2}=\QQ(\sqrt{13},\sqrt{-1}, \sqrt{-3})$.

\subsection{$L_{H_1,H_2}=\QQ(\sqrt{3}, \sqrt{-1}, \sqrt{2})(\sqrt{8(1+\sqrt{3})(4+2\sqrt{3})})$}\label{sub:H1H2}
It follows from Lemma \ref{lem:independence} that $H_1(\sigma)=\chi_{3}(\sigma)P_1$ and $H_2(\sigma)=\chi_{-1}(\sigma)P_1+\chi_{-2}(\sigma)P_2$ for all $\sigma\in \GalQ$, thus $H_1\cup H_2=\chi_{3}\cup \chi_{-2}$. Since $L_{H_1,H_2}$ is unramified outside $\{2,3,13\}$ and since $L_{H_1,H_2}$ contains a degree two extension $L$ of $\QQ(\sqrt{3},\sqrt{-2})$ such that $L/\QQ$ is Galois with Galois group $D_8$, same as in \ref{sub:jedan}, we can conclude that $L_{H_1,H_2}$ is governing field of $H_1$ and $H_2$. Moreover, for $p$ prime such that $H_1^{p}$ and $H_2^p$ define an element in $\Sel(E^{p})$ (or equivalently for prime $p$ which splits completely in $K_{H_1,H_2}=\QQ(\sqrt{3}, \sqrt{-1}, \sqrt{2})$), we have that $\CT{H_1^{p}}{H_2^p}$ is equal to $0$ if and only if $p$ splits completely in $L_{H_1,H_2}$.
\subsection{$L_{F_1,F_2}=\QQ(\sqrt{3}, \sqrt{-1}, \sqrt{2})(\sqrt{8(1+\sqrt{3})(4+2\sqrt{3})})$}
Here conclusion is the same as in \ref{sub:H1H2}, for $p$ prime such that $F_1^{p}$ and $F_2^p$ define an element in $\Sel(E^{-p})$ (or equivalently for prime $p$ which splits completely in $K_{F_1,F_2}=\QQ(\sqrt{3}, \sqrt{-1}, \sqrt{2})$), we have that $\CT{F_1^{p}}{F_2^p}$ is equal to $0$ if and only if $p$ splits completely in $L_{F_1,F_2}=L_{H_1,H_2}$.

\subsection{$L_{H,H_2}=\QQ(\sqrt{-1},\sqrt{2}, \sqrt{13})(\sqrt{4+2 \sqrt{13}})$}
 Lemma \ref{lem:independence} implies that $H(\sigma)=\chi_{13}(\sigma)P_2$ and $H_2(\sigma)=\chi_{-1}(\sigma)P_1+\chi_{-2}(\sigma)P_2$ for all $\sigma\in \GalQ$, thus $H\cup H_2=\chi_{13}\cup \chi_{-1}$.  Since $L_{H,H_2}$ is unramified outside $\{2,3,13\}$ and since $L_{H,H_2}$ contains a degree two extension $L$ of $\QQ(\sqrt{13},\sqrt{-1})$ such that $L/\QQ$ is Galois with Galois group $D_8$, same as in \ref{sub:jedan} we can conclude that $L_{H,H_2}$ is governing field of $H$ and $H_2$.
 Also, for $p$ prime such that $H^{p}$ and $H_2^p$ define an element in $\Sel(E^{p})$ (or equivalently for prime $p$ which splits completely in $K_{H,H_2}=\QQ(\sqrt{13}, \sqrt{-1}, \sqrt{2})$), we have that $\CT{H^{p}}{H_2^p}$ is equal to $0$ if and only if $p$ splits completely in $L_{H,H_2}$.
 
\subsection{$L_{H,H_1}=\QQ(\sqrt{3},\sqrt{13})(\sqrt{4+\sqrt{13}})$} 
  Lemma \ref{lem:independence} implies that $H(\sigma)=\chi_{13}(\sigma)P_2$ and $H_1(\sigma)=\chi_{3}(\sigma)P_1$ for all $\sigma\in \GalQ$, thus $H\cup H_1=\chi_{13}\cup \chi_{3}$.  Since $L_{H,H_1}$ is unramified outside $\{2,3,13\}$ and since $L_{H,H_1}/\QQ$ is $D_8$ extension same as in \ref{sub:jedan} we conclude that $L_{H,H_1}$ is governing field of $H$ and $H_1$. Also, for $p$ prime such that $H^{p}$ and $H_1^p$ define an element in $\Sel(E^{p})$, we have that $\CT{H^{p}}{H_1^p}$ is equal to $0$ if and only if $p$ splits completely in $L_{H,H_1}$.  Note that $H^{p}$ and $H_1^p$ are ELS if and only if $p=13$ or $p$ splits completely in $K_{H,H_1}=\QQ(\sqrt{13},\sqrt{3})$ and $p\equiv 1 \pmod{4}$.
  
\subsection{$L_{H^{-1},F_1}=\QQ(\sqrt{-2},\sqrt{13})(\sqrt{-1})$} 
 It follows from Lemma \ref{lem:independence} that for all $\sigma \in \GalQ$ we have that $H^{-1}(\sigma)=\chi_{13}(\sigma)P_1+\chi_{13}(\sigma)P_2=\chi_{13}(\sigma)P_3$ and $F_1(\sigma)=\chi_{-2}(\sigma)P_1+\chi_{-2}(\sigma)P_2=\chi_{-2}(\sigma)P_3$, thus $e_2(H^{-1}(\sigma),F_1(\sigma))=1$. Therefore $H^{-1}\cup F_1$ is alternating (see Lemma 3.1. in \cite{Smith}) and $\CT{H^{-d_1}}{F_1^{d_1}}=\CT{H^{-d_2}}{F_1^{d_2}}$ for all pairs $(d_1,d_2)\in \mathcal{D}$ from Theorem \ref{thm:Smith}.  For $p$ prime such that $H^{-p}$ and $F_1^p$ define an element in $\Sel(E^{-p})$, we can check by computing set $\mathcal{D}$ that $\CT{H^{-p}}{F_1^p}$ is equal to $0$ if and only if $p$ splits completely in $L_{H^{-1},F_1}$.  Note that $H^{-p}$ and $F_1^p$ are ELS if and only if $p$ splits completely in $K_{H^{-1},F_1}=\QQ(\sqrt{13},\sqrt{-2})$, thus, as before, the splitting behaviour of $p$ in $L_{H^{-1},F_1}$ determines Cassels-Tate pairing even though $L_{H^{-1},F_1}$ is not a governing field of $H^{-1}$ and $F_1$.
 
\section{Proofs of main results}
\begin{proof}[Proof of Theorem \ref{thm:main}]
From Section \ref{sec:governing} (see also Table \ref{table:fields}), we see that the governing field of the pair $(H^{-1},F_1)$ is $L_{H^{-1},F_1}=\QQ(\sqrt{-2},\sqrt{13})(\sqrt{-1})$. In particular, $$\CT{H^d}{F_1^{-d}}=\begin{cases} 
	0 &\textrm{ if $|d|$ splits completely in }L_{H^{-1},F_1},\\
	1 &\textrm{ otherwise }.
\end{cases}
$$
For $d<0$, it follows from the description of set $T$ that $\CT{H^d}{F_1^{-d}}=1$ if $d\equiv 1 \pmod{4}$ and $\CT{H^d}{F_1^{-d}}=0$ if $d\equiv 3 \pmod{4}$. Hence a) follows.
For b), assume that $d\equiv 3 \pmod{4}$ and $\iota(H^d) \ne 0$. As argued in the introduction, there is $L\in \Sel(E^d)$ such that $\CT{H^d}{L}=1$. Since $\CT{H^d}{F_1^{-d}}=0$, from the bilinearity of the Cassels-Tate pairing it follows that $\CT{H^d}{F_2^{-d}}=1$ (as $F_2$ is remaining generator of $\Sel(E^d)$). The other implication in b) is obvious.
Part c) is proved similarly. The only difference here is that in $d>0$ case, $\Sel(E^d)$ is, in addition to torsion classes, generated by $H^d, H_1^d$, and $H_2^d$.
\end{proof}

\begin{proof}[Proof of Corollary \ref{cor:limit}]
	First we count the contribution to $S(X)$ of $d=\pm p$ for which $d\notin T$. It follows from Conjectures \ref{conj:1} and \ref{conj:2}, and Propositions \ref{prop:infinity} and \ref{prop:sha_trivial} that the only significant case is when $w(E^d)=-1$ (assuming $H^d$ is ELS) in which case $\Sha(E^d)[2]$ is trivial. It follows from Propositions \ref{prop:ELS}, \ref{prop:root} and \ref{prop:sha_trivial} that this is equivalent to
	 $\legendre{d}{13}=1$, $\legendre{d}{2}\cdot\legendre{d}{3}\cdot\legendre{d}{13}=\textrm{sgn}(d)$ and $d \not \equiv 1 \pmod{8}$ if $d>0$ or $d\not \equiv 5,7 \pmod{8}$ if $d<0$. Thus if $$d \equiv 29, 35, 53, 55, 77, 79, 101, 103, 107, 127, 131, 155, 173, 179, 199, 251, 269, 295 \pmod{ 8\cdot 3\cdot 13}$$ when $d>0$ or if $d<0$ and
	  $$d \equiv 17, 43, 113, 139, 185, 209, 211, 233, 235, 257, 259, 283 \pmod{ 8\cdot 3\cdot 13},$$ then $H^d(\QQ) \ne \emptyset$. There are $18$ residue classes in the first case, and $12$ in the second, thus by Dirichlet's theorem on arithmetic progressions, the contribution to $C_1$ is $\frac{30}{2\phi(8\cdot 3\cdot 13)}=\frac{5}{32}$.
	  
	  Next, consider the case $d>0$, $d\in T$ and $\Sha(E^d)[2]\ne 2\Sha(E^d)[4]$.
	  Corollary \ref{cor:main} together with Proposition \ref{prop:sha24} implies that in this case $H^d(\QQ)\ne \emptyset$ if and only if 
	  $d$ does not split completely in $L_{H_1,H_2}$, and splits completely in $L_{H,H_1}$ and $L_{H,H_2}$. One can check that the assumption $d>0$ and $d\in T$ is equivalent to the requirement that $d$ splits completely in $K_{H,H_1}$, $K_{H,H_2}$ and $K_{H_1,H_2}$, thus we need to find a density of $d$'s such that $d$ splits completely in composition $K = L_{H,H_1} L_{H,H_2} K_{H_1,H_2}$ but not in its degree two extension $L= L_{H,H_1} L_{H,H_2} L_{H_1,H_2}$. By Chebotarev density theorem the density of such $d$'s is $\frac{1}{\deg{K}} \cdot \frac{1}{2}$.
	  From Table \ref{table:fields} we see that $K_{H_1,H_2}$ is contained in  $L_{H,H_1} L_{H,H_2}$. Moreover, one can check that $\deg{ L_{H,H_1} L_{H,H_2}}=64$, thus in this case the contribution to $C_1$ is equal to $\frac{1}{2}\cdot\frac{1}{128}$ (we have extra $\frac{1}{2}$ since $C_1$ is a lower bound for $\frac{S(X)}{2\pi(X)}$ and not $\frac{S(X)}{\pi(X)}$ ).
	  
	  Finally, consider the case $d<0$, $d\in T$ and $\Sha(E^d)[2]\ne 2\Sha(E^d)[4]$.
	  Corollary \ref{cor:main} together with Proposition \ref{prop:sha24} implies that $H^d(\QQ)\ne \emptyset$ if and only if
	  $d=-p$, where $p\equiv 1 \pmod{4}$, does not split completely in $L_{F_1,F_2}$ and splits completely in $L_{H^{-1},F_2}$. One can check that assumption $p\equiv 1 \pmod{4}$ and $-p \in T$ is equivalent to $p$ splits completely in $K_{H^{-1},F_2}$ (we see in Table \ref{table:fields} that $\QQ(\sqrt{-1}) \subset K_{H^{-1},F_2})$ and $K_{F_1,F_2}$. As in the previous case, we need to compute the density of primes which split completely in composition $L_{H^{-1},F_2} K_{F_1,F_2}$, but not in its degree two extension  $L_{H^{-1},F_2} L_{F_1,F_2}$.
	  Since $\deg{L_{H^{-1},F_2} K_{F_1,F_2}}=32$, in this case the contribution to $C_1$ is equal to $\frac{1}{2}\cdot\frac{1}{64}$. Hence it follows that $C_1=\frac{5}{32}+\frac{1}{256}+\frac{1}{128}=\frac{43}{256}$.
	  
	  To compute the upper bound $C_2$, we need to find the density of the remaining case, $d\in T$ and $\Sha(E^d)[2] = 2\Sha(E^d)[4]$, in which our method does not provide us an answer. If $d>0$, by Proposition \ref{prop:sha24} it is enough to compute the density of primes $p$ which splits completely in $L_{H,H_1}$, $L_{H,H_2}$ and $L_{H_1,H_2}$.  From Table \ref{table:fields}, we see that the composition of these three fields have degree $128$, hence by Chebotarev density theorem the density of primes with this splitting property in $1/128$, hence contribution to $C_2-C_1$ is $1/256 $. 
	  	
	  	If $d<0$ and $d\equiv 7 \pmod{8}$, then $p$ must split completely in $L_{F_1,F_2}, L_{H,F_2}$ and $K=\QQ(\sqrt{-2},\sqrt{13})$ (see Table \ref{table:fields}), and furthermore it must either split completely in $$L=\QQ(\sqrt{-2},\sqrt{13})(\sqrt{4+2\sqrt{13}})$$ or none of its factors in $K$ splits further in $L$ (note that $L/\QQ$ is not Galois extension). One can check that this condition is equivalent for $p$ to split completely in composition $L_{F_1,F_2} L_{H,F_2}$ which is of degree $64$, hence the density of such primes is $1/64$, and contribution to $C_2-C_1$ is equal to $1/128$. Hence $C_2=C_1+1/256+1/128=46/256$.
\end{proof}

\section{Future work}
This paper left us with some interesting questions which may be addressed in the future projects:
\begin{enumerate}
	\item [a)] What information can be obtained about $H^d(\QQ)$ in the case when $\Sha(E^d)[2] = 2\Sha(E^d)[4]$?
	\item [b)] What can one say about $H^d(\QQ)\ne \emptyset$ for some larger class of $d$'s? The main reason why we considered only $d$'s for which $|d|$ is prime is that in this case we can control the $2$-Selmer group of quadratic twists $E^d$ - we have explicit generators. This might also be the case, for example, for the set of $d$'s which are the products of two primes. 
	\item[c)] Can one obtain similar results for the quartics other that $H$? It seems this could be within the reach of this method provided that, as in b), we have explicit description of $2$-Selmer groups of quadratic twists.
\end{enumerate}

\section*{Acknowledgments}
The author was supported by the Croatian Science Foundation under the project no.~IP-2018-01-1313, and by the QuantiXLie Center of Excellence, a project co-financed by the Croatian Government and European Union through the European Regional Development Fund - the Competitiveness and Cohesion Operational Programme (Grant KK.01.1.1.01.0004).

\bibliographystyle{alpha}
\bibliography{bibliography}
\end{document}